\documentclass{llncs}

\usepackage[utf8]{inputenc}
\usepackage{enumerate}
\usepackage[shortlabels]{enumitem}
\usepackage{graphicx}
\usepackage{epstopdf}
\usepackage{amsfonts}
\usepackage{amssymb}
\usepackage{amsmath}
\usepackage[labelfont=bf]{caption}
\usepackage{sectsty}
\usepackage{wasysym}
\usepackage{hyperref}
\usepackage[usenames,dvipsnames]{xcolor}  % typesetting in color
\usepackage{algorithmic}
\usepackage{algorithm}

\newcommand{\N}{\mathbb{N}}
\newcommand{\Z}{\mathbb{Z}}

\newcommand{\rng}{\mathrm{r}}
\newcommand{\avg}{\mathrm{\overline{r}}}
\newcommand{\maxr}{\mathrm{r}^{\max}}
\newcommand{\diam}{\mathrm{diam}}

\newcommand{\lip}{\mathcal{L}}
\newcommand{\abs}[1]{\left|{#1}\right|}

\def\computationproblem#1#2#3{% {problem_name}{input}{output}
  \begin{center}
  \begin{tabular}{rp{0.8\textwidth}}
  {\sc Problem:\enspace}&#1\\
  {\sc Input:\enspace}&#2\\
  {\sc Question:\enspace}&#3\\
  \end{tabular}
  \end{center}
}

\begin{document}

\title{Algorithmic aspects of\\ $M$-Lipschitz mappings of graphs}

\author{Jan Bok}

\institute{
Computer Science Institute, Charles University, Malostransk\'{e} n\'{a}m\v{e}st\'{i} 25, 11800, Prague, Czech Republic, email: \email{bok@iuuk.mff.cuni.cz}.
}

\maketitle

\begin{abstract}
$M$-Lipschitz mappings of graphs (or equivalently graph-indexed random walks)
are a generalization of standard random walk on $\mathbb{Z}$. For $M \in \N$,
an \emph{$M$-Lipschitz mapping} of a connected rooted graph $G = (V,E)$ is a
mapping $f: V \to \Z$ such that root is mapped to zero and for every edge
$(u,v) \in E$ we have $|f(u) - f(v)| \le M$.

We study two natural problems regarding graph-indexed random walks.
\begin{enumerate}
\item Computing the maximum range of a graph-indexed random
walk for a given graph.
\item Deciding if we can extend a partial GI random walk into
a full GI random walk for a given graph.
\end{enumerate}

We show that both these problems are polynomial-time solvable and we show
efficient algorithms for them. To our best knowledge, this is the first
algorithmic treatment of Lipschitz mappings of graphs. Furthermore, our
problem of extending partial mappings is connected to the problem of
\emph{list homomorphism} and yields a better run-time complexity for a
specific family of its instances.
 \end{abstract}

\section{Introduction}

Graph-indexed random walks (or also $M$-Lipschitz mapping of graphs) are a
generalization of standard random walk on $\mathbb{Z}$. This concept has
important connections to statistical physics, namely to gas models and
\emph{Widom–Rowlinson configurations} (as is described by Zhao
\cite{zhao2017extremal} and Cohen et al.\ \cite{cohen2017widom}). For a more
general treatment of random walks, see
\cite{lovasz1993random,haggstrom2002finite,lawler2010random}.

Graph-indexed random walks were thoroughly studied, mainly because
of the parameter of the average range, for example in \cite{benjamini2000random,erschler2009random,benjamini2000upper,wuaverage,loebl2003note}.
However, we emphasize that algorithmic aspects of graph-indexed random walks were,
by our best knowledge, not studied yet. 

\paragraph*{Applications and motivation.}

We believe that our results can be useful in determining the complexity
of computing the average range and in statistical physics.

Results on finding the maximum range provide an easy tool to determine
the extremal cases of graphs with regard to the number of $1$-Lipschitz
mappings. Furthermore, one can ask if there is some $M$-Lipschitz mapping
for a given $M$ and a given graph $G$ with $k \in \Z$ as the image of
a vertex in $G$. Results in Section 3 provide a
clear positive answer to this. We can check this in linear time.

Results on extending partial Lipschitz
mappings are motivated by the following.
\begin{itemize}
  \item Our algorithms for extending partial Lipschitz mappings provide
  faster algorithms for particular instances of list homomorphism problem
  compared to algorithms of Feder and Hell in \cite{feder1998list,feder1999list}.
  \item We often deal with incomplete or corrupted data. Finding if some given
  mapping can be a part of an $M$-Lipschitz mapping can be seen as
  a quick routine to exclude cases of clearly inconsistent data.
\end{itemize}

\section{Preliminaries}

We use the standard notation and definitions as in Diestel's monograph
\cite{diestel2000graph}. Intervals in this paper are closed intervals of integers,
if not stated otherwise.

A \emph{graph homomorphism} between digraphs $G$ and $H$ is a
mapping $f: V(G) \to V(H)$ such that for every edge $uv \in E(G)$, $f(u)f(v)
\in E(H)$. That means that graph homomorphism is an adjacency-preserving
mapping between the vertex sets of two digraphs. The set $I := \{ w \in V(H) \mid
\exists v \in V(G): f(v) = w  \}$ for a graph homomorphism $f$ is
called the \emph{homomorphic image} of $f$.

For a comprehensive and more complete source on graph homomorphisms, the
reader is invited to see \cite{hell2004graphs}. A quick introduction is given
in \cite{godsil2013algebraic} as well.

\begin{definition} \label{def:lipschitz}
For $M \in \N$, an \emph{$M$-Lipschitz mapping} of a connected graph $G = (V,E)$ with root $v_0 \in V$ is a
mapping $f: V \to \Z$ such that $f(v_0) = 0$ and for every edge $(u,v) \in E$
it holds that $|f(u) - f(v)| \le M$. The set of all $M$-Lipschitz mappings of a
graph $G$ is denoted by $\mathcal{L}_M(G)$.
\end{definition}

We strongly emphasize that we are interested in connected graphs only. Components
without the root would also allow infinitely many new $M$-Lipschitz mappings.
In case of disconnected graphs we can apply a suitable linear
transformation (for example $x \to x+1$) to images of vertices we would
get a new $M$-Lipschitz mapping.

The root is just some distinguished vertex of $G$. The reason for having
graphs rooted is that we want to have finitely many Lipschitz mappings for a
fixed graph $G$. One can reason similarly as in the case of disconnected graphs.

In literature, we will often meet a slightly different definition of
$1$-Lipschitz mappings. In it the restriction $|f(u) - f(v)| \le 1$, for all
$uv \in E$, is removed and instead, the restriction $|f(u) - f(v)| = 1$, for all
$uv \in E$, is added. In \cite{loebl2003note} authors call these mappings
\emph{strong Lipschitz mappings}. We generalize this in the following
definition.

\begin{definition}
For $M \in \N$, a \emph{strong $M$-Lipschitz mapping} of a connected graph $G = (V,E)$ with root $v_0 \in V$ is a
mapping $f: V \to \Z$ such that $f(v_0) = 0$ and for every edge $(u,v) \in E$
it holds that $|f(u) - f(v)| = M$. The set of all $M$-Lipschitz mappings of a
graph $G$ is denoted by $\mathcal{L}_{\pm M}(G)$.
\end{definition}

Note that strong $M$-Lipschitz mappings are a special case of $M$-Lipschitz mappings
of graph. We further emphasize the following lemma.

\begin{lemma} \label{lem:strongchar}
  A graph has a strong $M$-Lipschitz mapping iff it is bipartite.
\end{lemma}

We now define the main parameters for Lipschitz mapping of graphs.

\begin{definition}
The \emph{range} of a Lipschitz mapping $f$ of $G$ is the size of the
homomorphic image of $f$. Formally
$\rng_G(f) := \big|\{ z \in \Z\,|\,z = f(v) \textrm{ for some } v \in V(G) \}\big|$.
\end{definition}

\begin{definition} \emph{(Average range)}
The \emph{average range} of graph $G$ over all $M$-Lipschitz mappings is defined as
$$\avg_M(G) := \frac{\sum_{f \in \mathcal{L}_M(G)} \rng(f)}{|\mathcal{L}_M(G)|}.$$
\end{definition}

We can view this quantity as the expected size of the homomorphic image of an
uniformly picked random $M$-Lipschitz mapping of $G$.

\begin{definition} \emph{(Maximum range)} \label{def:maximum_range}
The \emph{maximum range} over all $M$-Lipschitz mappings of graph $G$ is defined as
$$\maxr_M(G) = \max_{f \in \mathcal{L}_M(G)} \rng(f).$$
\end{definition}

Whenever we want to talk about the counterparts of these definitions for
strong Lipschitz mappings, we denote it with $\pm$ in subscript. For example,
$\avg_{\pm M}$ is the average range of strong $M$-Lipschitz mapping of graph.

Whenever we write average range or maximum range without saying which
$M$-Lipschitz mappings we use, it should be clear from the context what $M$ do
we mean.

It is worth noting that for computing the average range and the maximum range,
the choice of root does not matter. That is why we often omit the details of
picking a root.

\paragraph*{Connection to graph homomorphisms.}
$M$-Lipschitz mappings map graph vertices to integers. There is a
natural bijection between $M$-Lipschitz mappings and graph homomorphisms to a
suitable graph associated with $\Z$. Consider a graph $Z_M$ with the vertex set
$V(Z_M)=\Z$ and the edge set $E(Z_M) = \{ij: \abs{i - j} \leq
M \}$. Every $M$-Lipschitz mapping corresponds to a graph homomorphism to $Z_M$.

We can define a graph $Z_{\pm M}$ analogously for strong mappings.
Note that in the case of strong $1$-Lipschitz mappings,
we get that they correspond to homomorphisms to a two-way infinite path and in the case of $1$-Lipschitz mappings,
we get a correspondence to homomorphisms to a two-way infinite path with loops added to each vertex.
See Figure \ref{fig:homtoz} for an example of such
homomorphism of $C_4$.

\begin{figure}
\centering
\includegraphics[scale=0.4]{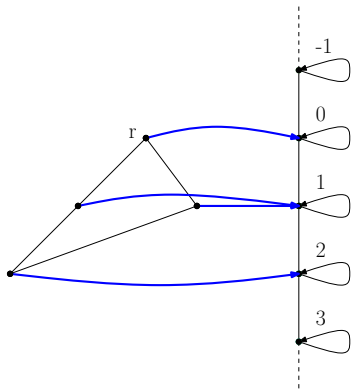}
\caption{A homomorphism of $C_4$ rooted in $r$ to $Z_1$ graph. In fact, this homomorphism
is a $1$-Lipschitz mapping of $C_4$.}
\label{fig:homtoz}
\end{figure}

\paragraph*{Gas models and physical motivation.}

A~homomorphism from G to $P_3$ with loops added to
each vertex corresponds to a partial (not necessarily proper) coloring of the
vertices of G with red or blue, allowing vertices to be left ``uncolored''
such that no red vertex is adjacent to a blue vertex. This coloring is known
as the \emph{Widom–Rowlinson configuration} \cite{widom1970new} 
Observe that Widom-Rowlinson configuration
corresponds to a $1$-Lipschitz mapping with the size of the homomorphic image
at most 3. See Figure \ref{fig:model} for an example.

Widom-Rowlinson configurations have a physical interpretation. Consider
particles of a gas $B$ (blue vertices) and of a gas $R$ (red vertices).
W-R configurations then model situations in which particles of gases $A$ and
$B$ do not interact. This model is sometimes referred to as the \emph{
hard-core model}. The name emphasizes the hard restriction that
particles of gases cannot be directly adjacent, i.e.\ their molecules do not
interact.

\begin{figure}[h!]
\centering
\includegraphics[scale=0.7]{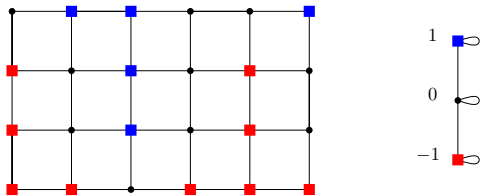}
\caption{An example of Widom–Rowlinson configuration for a grid graph.}
\label{fig:model}
\end{figure}

\section{Maximum range} \label{chap:maximum_range}

In this section we will show how can we algorithmically compute the maximum range
of a given graph.
Also, we will show the relation of this parameter to other existing results.

\subsection{Diameter}

In this section we observe one important fact giving us an upper bound on the
range of a graph. Then we will show that this upper bound is tight.

We will first prove an important, yet easy lemma.

\begin{lemma} \label{lem:diam}
For any connected graph $G$ with $\diam(G)$ and every $M$-Lipschitz mapping
$f$ of $G$, holds that $\rng(f) \le M \cdot (\diam(G) + 1).$
\end{lemma}

Now we show that we can always construct a mapping where equality holds and thus
we conclude that the diameter and the maximum range are tightly connected.

\begin{theorem} \label{thm:diam}
For any connected graph $G$, $\maxr_M(G) = M \cdot (\diam(G) + 1)$.
\end{theorem}
\begin{proof}
From the definition of the diameter, there must exist vertices $u_1$ and $u_2$ such that their distance
is equal to $\diam(G)$. Without loss of generality we set $r := u_1$.
Now let us define the mapping $f: V(G) \to \Z$ so that for every $v \in V$ we have $f(v) := M \cdot d(r,v)$.

We see that $f(r)=0$, and $f(u_2) = M \cdot d(r,u_2)$ so the image of the
shortest path
connecting $u_1$ and $u_2$ has the size $\diam(G) + 1$. On the top of that,
for every $uv \in E(G)$, $|f(u)-f(v)| \le M$, otherwise we would get a
contradiction with the definition of the distance. Thus $f$ is an $M$-Lipschitz
mapping and its range has to be at least $M \cdot (\diam(G) + 1)$. Combining
this with Lemma \ref{lem:diam}, we get the claim weT wanted to prove.
\qed \end{proof}

\subsection{The case of strong mappings}

By Lemma \ref{lem:strongchar} we showed that strong Lipschitz mappings can
exist on bipartite graphs only. We will now extend Theorem \ref{thm:diam}.

\begin{theorem} \label{thm:diamstrong}
For any bipartite connected graph $G$, $\maxr_{\pm M}(G) = M \cdot (\diam(G) + 1)$.
\end{theorem}

\subsection{Applications}

We will apply our results to prove Theorem \ref{thm:number} and subsequently
to prove extremal behavior on the number of Lipschitz mappings of a graph.

\begin{theorem} \label{thm:number}
For every connected graph $G=(V,E)$ and for every two vertices
$a,b \in V$ such that $ab \not\in E$, holds that
$$\abs{\lip_1(G)} \ge \abs{\lip_1(G \cup \{a,b \})}.$$
\end{theorem}

We will use the Cherry lemma for the proof.

\begin{lemma}[Cherry lemma] \label{lem:cherry} 
  A graph $G$ is a disjoint union of complete graphs if and only if it
  does not contain $K_{1,2}$ as an induced subgraph.
\end{lemma}

Now we can prove Theorem \ref{thm:number}.

\begin{proof}[Proof of Theorem \ref{thm:number}]
  The graph $G$ cannot be a complete graph. Therefore, by Lemma \ref{lem:cherry},
  induced $K_{1,2}$ exists in $G$. Let vertices $a$ and $b$ from the
  statement of Theorem \ref{thm:number} be the two non-adjacent vertices of induced $K_{1,2}$.

  We see that $G$ has the diameter at least $2$, since $a$ and $b$ are in distance
  $2$. Let us root $G$ in $a$ for auxiliary reasons.   By the construction of $1$-Lipschitz mapping from Theorem \ref{thm:diam}, there must exist a mapping $f$ with $f(b) = d(a,b) = 2$.

  Clearly, $\lip_1(G \cup ab) \subseteq \lip_1(G)$. However, $f$ cannot be
  a $1$-Lipschitz mapping of $G \cup ab$ rooted in $a$. That implies 
  $|\lip_1(G \cup ab)| \leq |\lip_1(G)| - 1$.
\qed \end{proof}

Theorem \ref{thm:number} further implies the following theorem, giving extremal
graphs with respect to the number of Lipschitz mappings.

\begin{corollary}
  Among connected graphs of order $n$, trees have the maximum
  number of $1$-Lipschitz mappings and a complete graph $K_n$ has the
  minimum number of $1$-Lipschitz mappings.
\end{corollary}

\subsection{Algorithmic aspects}

Let us consider the following algorithmic problems -- \textsc{$M$-MaxRange} and
\textsc{$M$-Strong-MaxRange}.

\computationproblem
{Maximum range problem -- \textsc{$M$-MaxRange}}
{A connected graph $G$.}
{What is the maximum range of $M$-Lipschitz mapping of
$G$, i.e.\ $\maxr_M(G)$?}

The problem \textsc{$M$-Strong-MaxRange} can be defined similarly.

Because of Theorem \ref{thm:diam}, we can use the existing algorithms for finding
graph diameter and distance in graphs for both of these problems. 
We can achieve an even better
complexity for some classes. Take for example the class of trees in which we can compute
diameter by a linear-time algorithm using one clever depth-first search traversal.

\section{Extending partial Lipschitz mappings} \label{chap:extending}

While studying Lipschitz mappings of graphs we came up with an algorithmic problem which
falls into widely studied paradigm of a partial structure extension. We give
two examples of such problems to show a broader context.

\subsection{Related problems}

\paragraph*{Precoloring extension.}
The following problem was introduced in the series of papers \cite{biro1992precoloring,hujter1993precoloring,hujter1996precoloring}.

\computationproblem
{\textsc{Precoloring Extension}}
{An integer $k \geq 2$, a graph $G=(V,E)$ with $|V|\geq k$, a vertex
subset $W \subseteq V$, and a proper $k$-coloring of $G_W$.}
{Can this $k$-coloring be extended to a proper $k$-coloring of the whole
graph $G$?}

To current date, more than twenty papers on the precoloring extension problem were
published. No up-to-date survey is available, but Daniel Marx gathers
an unofficial list of relevant papers on his webpage \url{http://www.cs.bme.hu/~dmarx/prext.php}.

% zminit jeste Tuzuv survey

%The problem originally arised in Colbourn's paper on completing partial
%Latin squares \cite{colbourn1984complexity}.

\paragraph*{The partial representation extension problem.}
The reader surely
knows a planar drawing of graph. A particular drawing of the underlying graph can be
seen as one of the possible representations. Studying the representations of various graph
classes is a wide area of graph theory and we refer reader to the comprehensive
monograph of Spinrad \cite{spinrad2003efficient}. One can ask for a given
graph $G$ and some partial representation $R'$ of $G$ if it can be extended to some
full representation $R$ of $G$ such that $R' \subseteq R$. This problem was studied
for various graph classes, see PhD thesis of Klavík
\cite{klavikphd} for a survey.

\subsection{Definition of our problem}

We will define two similar problems in the setting of integer homomorphisms.

\computationproblem
{Partial $M$-Lipschitz mapping extension - \textsc{$M$-LipExt}}
{A connected graph $G=(V,E)$, a subset $V' \subseteq V$ with a function $f': V' \to \Z$.}
{Does there exist an $M$-Lipschitz mapping $f$ of $G$ such that $f' \subseteq f$?}

The problem \textsc{Strong $M$-LipExt} can be defined similarly.

If the answer for a given instance of \textsc{$M$-LipExt} (or
\textsc{Strong $M$-LipExt}) is YES, we say that $f'$ is \emph{extendable} for the
given $G$ and the given type of problem. We often say only that $f'$ is extendable when it is clear from the
context which of these two problems we are trying to solve.

See Figure \ref{fig:digraph} for an initial example. This mapping cannot
be extended to a $1$-Lipschitz mapping but it can be extended to an $L$-Lipschitz
mapping for every $L \geq 2$.

\begin{figure}
\centering
\includegraphics[scale=0.7]{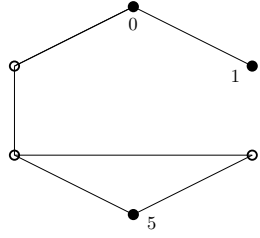}
\caption{An example of a partial mapping with three prescribed vertices.}
\label{fig:digraph}
\end{figure}

\subsection{Partial non-strong $M$-Lipschitz mappings}

\paragraph*{Polynomiality.}

We will show that \textsc{$M$-ParExt} can be polynomially reduced to a tractable instance
of list homomorphism problem.

\computationproblem
{List homomorphism problem - \textsc{LHom($H$)}}
{A graph $G$ and a list function $L: V(G) \to 2^{V(H)}$.}
{Does there exist an homomorphism $f: G \to H$ such that $f(v) \in L(v)$ for every
$v \in V(G)$?}

\begin{theorem}\label{thm:parext}
  The problem \textsc{$M$-ParExt} is polynomial-time solvable.
\end{theorem}
\begin{proof}{{(Sketch.)}}
Without loss of generality, assume that $f'(V') \subseteq [-n,n]$. We define $n:=|V(G)|$. For a given instance of \textsc{$M$-ParExt} instance, build a graph $T$ with $V(T):=\{-n,\ldots,0,\ldots n\}$
and $E(T):= \{ (a,b):  -n \le a \le b \le n \wedge |a-b| \le M \}$. If $v \in V'$,
set $L(v) := \{f'(v)\}$ and if $w \in V\setminus V'$, set $L(w) := V(T)$. Now observe
that answer for \textsc{LHom($T$)} with $G$ on input is positive if and only if
 \textsc{$M$-ParExt} on $G$ has a positive answer.
  
  Feder and Hell \cite{feder1998list} proved that if $H$
  has a loop on each vertex and $H$ is an interval graph, then
\textsc{LHom($H$)} is tractable. Hence what remains is to prove that $T$ is an interval graph. Our result then follows.
\qed \end{proof}

Note that the algorithm in \cite{feder1998list} runs in $O(|V|^4)$ for the case
of \textsc{LHom($Z_M$)}
Our goal in the remaining text will be to show more efficient algorithms for
these instances. We note
that our approach will be also constructive, as is the algorithm in \cite{feder1998list}, 

\paragraph*{Trees.} 

The goal of this part of the paper is to show that we can solve
\textsc{$M$-ParExt} in quadratic time and linear space with
a special algorithm. We will now prove the correctness and the complexity of this algorithm.

\begin{algorithm}[h]
\caption{Algorithm for \textsc{$M$-ParExt} on trees.}
\label{alg:trees}
\begin{algorithmic}[1]
\REQUIRE A tree graph $G$, a vertex set $V' \subseteq V(G)$, and a partial $M$-Lipschitz mapping $f': V' \to \Z$.
\smallskip

\STATE Check if $|f'(v)-f'(u)| \leq M$ for all $u,v \in V'$. If not, $f'$ cannot be extended.
\STATE Set $P(v) := [f'(v),f'(v)]$ for every $v \in V'$, 
\STATE Set $P(v) := [-\infty,\infty]$ for every $v \in V(G)\setminus V'$.
\smallskip
\FOR{every $v'$ in $V'$}
  \STATE Start the DFS on $G$ from $v'$.
  \STATE In DFS, when you process vertex $v$ with $P(v) = [\underline{P}(v),\overline{P}(v)]$, do the following:
    \FOR{every $w \in N_G(v)$}
    \STATE $P(w) := [\underline{P}(v) - M, \overline{P}(v) + M] \cap P(w)$.
    \ENDFOR
\ENDFOR
\smallskip
\STATE Find $r \in V(G)$ such that $0 \in P(v)$ and re-run DFS from Line 5 with $v' = r$.
\IF{no such r exists}
\RETURN The mapping $f'$ cannot be extended.
\ENDIF
\STATE Set $f(r) := 0$.
\smallskip
\IF{$P(v) = \emptyset$ for some $v \in V(G)$}
\RETURN The mapping $f'$ cannot be extended.
\ENDIF
\smallskip
\STATE Launch the BFS from $r$ and for every visited vertex $v \neq r$, set $f(v)$ so that for parent vertex $p$, $f(v) \in [f(p)-M,f(p)+M]$ holds.
\smallskip
\IF{the previous BFS could not be completed}
\RETURN \FALSE
\ENDIF
\smallskip
\RETURN \TRUE
\end{algorithmic}
\end{algorithm}

\begin{lemma}[Correctness]\label{lem:treescorrect}
  Algorithm \ref{alg:trees} is correct. It finds an $M$-Lipschitz
  mapping $f$ that extends $f'$ if and only if $f'$ is extendable.
\end{lemma}
\begin{proof}
  Suppose that the algorithm returns a mapping $f$. We claim that it is
  an $M$-Lipschitz mapping extending $f'$. Obviously, there exists a vertex
  mapped to zero under $f$ -- the vertex $r$. Furthermore, the condition
  $$|f(u)-f(v)| \leq M, \forall uv \in E(G)$$
  holds, otherwise the algorithm would stop on Line 20. Finally, we
  observe that for every $v' \in V'$, interval $P(v')$ is
  equal to $[f'(v'),f'(v')]$ at the end of the algorithm so $f$ extends $f'$.
  That finishes the only if part of the equivalence.

  Now let us prove the if part. We will prove that if the algorithm
  does not find an $M$-Lipschitz mapping $f$ extending $f'$, then $f'$ is not extendable.

  Algorithm can stop and fail to find such $f$ exactly from the following reasons:
  \begin{enumerate}
    \item \textit{Algorithm could not find a candidate for the root. (Line 12)}

    If at the end of the algorithm for every vertex $v \in V$, $0 \not\in P(v)$,
    then for every $v \in V$ exists some vertex $v' \in V'$ such
    that $|f(v')| > M \cdot d(v,v')$. We see that $f'$ is not extendable.

    \item \textit{There exists $v \in V$ such that $P(v) = \emptyset$. (Line 16)}

    If such $v$ exists, then it implies the existence of two vertices
    $c,d \in V'$ such that the intersection $I = [f'(c) - M \cdot (c,v), f'(c) + M \cdot (c,v) ]  \cap [f'(d) - M \cdot (d,v), f'(d) + M \cdot (d,v) ]$ is empty.
    However, $I$ is exactly the set of all possible images that we can assign
    to $v$ if $c$ is set to $f'(c)$ and $d$ is set to $f'(d)$. We conclude
    that $f'$ is not extendable.

    \item \textit{Algorithm could not complete the BFS phase. (Line 20)}

    We will actually show that this case is not possible since the only
    possibility how 3) can happen is that some final interval $P(v)$ for some $v \in V$
    is empty and the algorithm will halt even before the BFS phase can start
    (more precisely, the algorithm would already stop at line 16 and the case 2) occurs).

    Assume that all intervals $P(v)$ are nonempty. Consider an edge $xy \in E(G)$.
    Without loss of generality, in the last DFS phase (Line 11),
    $x$ was processed before $y$. Consider intervals $P'(x), P'(y)$ defined
    as the intervals $P(x),P(y)$, respectively, before the last DFS phase. Clearly,
    when $x$ was processed in the last DFS phase, $P(y)$ was set to $P'(y) \cap [P'(x)-M,P'(x)+M]$; a nonempty interval and therefore,
    $$\forall i \in P(x), \exists j \in P(y): |i-j| \leq M.$$
    And conversely,
    $$\forall j \in P(y), \exists i \in P(x): |i-j| \leq M.$$
    We conclude that the case 3) cannot occur.
  \end{enumerate}

  This proves the if part and we are done.
\qed \end{proof}

 We can now conclude the main theorem for trees.

\begin{theorem}
  \textsc{$M$-ParExt} for trees is solvable in time $O(|V|^2)$ and space $O(|V|)$.
\end{theorem}
\begin{proof}
  We proved that Algorithm \ref{alg:trees} is correct.

  We are running $O(|V|)$ times depth-first search on $G$ plus we perform a constant number
  of linear traversals of data structure for $G$. That, combined with $G$ being a tree, concludes the claim.
\qed \end{proof}

\paragraph*{General case.}
The goal of this section is to show an efficient constructive algorithm for the general case.

\begin{theorem} \label{thm:general}
  The problem \textsc{$M$-ParExt} is solvable in polynomial time on general graphs.
  There is an algorithm for it running in time $O(|V|^3)$ and $O(|V|^2)$ space.
  Furthermore, if an instance for \textsc{$M$-ParExt} is rooted, we have an
  algorithm running in time $O(|V||E|)$ and space $O(|V|^2)$.
\end{theorem}

It will be useful to define two new properties for integer functions on vertex sets.

\begin{definition}[$M$-reachability]
  We call a mapping $f: V' \to \mathbb Z$ of graph $G$, with $V' \subseteq V(G)$,
  \emph{$M$-reachable} if every pair of vertices $u,v \in V'$ satisfies
  $$\abs{f(u)-f(v)} \leq M \cdot d(u,v),$$

  We omit for which graph is a mapping $M$-reachable if it clear from context.

  %We say that a vertex $v \in V'$ is \emph{$M$-reachable} if for every vertex
  %$u \in V'$ holds $$\abs{f(u)-f(v)} \leq M \cdot d(u,v).$$
\end{definition}

\begin{definition}
  We call a mapping $f: V' \to \mathbb Z$ of graph $G$, with $V' \subseteq V(G)$,
  \emph{rooted} if $f^{-1}(0) \neq \emptyset$.
\end{definition}

We can now state and prove the full characterization of extendable situations.

\begin{theorem} \label{thm:ext}
  For a graph $G=(V,E)$, subset $V' \subseteq V$, and a partial mapping $f': V' \to \mathbb{Z}$, the following statements are equivalent:
  \begin{enumerate}
    \item The mapping $f'$ is extendable to a $M$-Lipschitz mapping.
    \item One of the following holds:
    \begin{enumerate}
    \item The mapping $f'$ is $M$-reachable and rooted.
    \item There exist $r \in V \setminus V'$ with $f''$ defined as
      $f'' := f' \cup (r,0)$, such that $f''$ is $M$-reachable.
    \end{enumerate}
  \end{enumerate}
\end{theorem}
\begin{proof}
    $(1)\implies(2)$: For brevity, we will handle the (a) case, (b) case is similar.
  Assume that $f'$ is extendable to an $M$-Lipschitz mapping $f^*$.
  Choose a pair of vertices $u,v \in V'$ such that $\abs{f(u)-f(v)} > M \cdot d(u,v)$.
  Pick a shortest path $u=x_1,\ldots,x_l=v$ between $u$ and $v$. Each summand of the sum
  $\sum_{i=2}^{l}(f(x_i)-f(x_{i-1}))$ is at most $M$ and at least $-M$, which contradicts
  that the mapping $f^*$ is $M$-Lipschitz.

  $(2)\implies(1)$: Again, we will only prove the (a) case, (b) can be done
  in analogous way. We will show that we can extend $f'$ mapping by one vertex
  and preserve $M$-reachability. By applying this inductively we will prove
  that $(2)$ implies $(1)$.

  Choose a vertex $a$ that is adjacent to some vertex $b$ such that $f'(b)$ is
  defined and $f'(a)$ is not defined. For every vertex $c \in f'(V')$,
  the vertex $b$ is reachable. Thus we can always find a number for $a$ such that
  $c$ will be reachable from $a$. For every $c$ we can define interval $I_c$
  containing all possible values for $a$ such that $c$ is reachable from $a$. This
  is a closed connected interval in $\mathbb{Z}$. Furthermore, the set
  system $\{I_c | \forall c \in f'(V')\}$ has the Helly property.
  Finally, for every two $c_1,c_2 \in f'(V')$ there is a nonempty intersection.
  Otherwise we would get contradiction with $M$-reachability of $f'$.
  Thus we can pick a suitable $k \in \Large\cap_{c \in f'(V')} I_c$ and extend
  $f'$ by setting $f' := f' \cup \{ (a,k) \}$. We also set $V' := V' \cup a$.
\qed \end{proof}

Theorem \ref{thm:ext} and its proof yield an algorithm for constructing an
extended mapping. Now we can finally prove Theorem \ref{thm:general}.

\begin{algorithm}[h]
\caption{A constructive algorithm for \textsc{$M$-ParExt} on general graphs.}
\label{alg:general}
\begin{algorithmic}[1]
\REQUIRE A graph $G$, a set of vertices $V' \subseteq V(G)$, and a partial $M$-Lipschitz mapping $f': V' \to \Z$.
\smallskip
\STATE Compute all-pairs distances using Thorup's algorithm \cite{Thorup}.
\IF{$f'$ rooted} 
\STATE Set $f'':=f'$ and go to line 8.
\ENDIF
\IF{$f'$ not rooted}
\FOR{every $v'$ in $V\setminus V'$}
  \STATE Set $f'' := f' \cup (v',0)$.
  \IF{$f''$ is $M$-reachable}
  \WHILE{some vertex not mapped under $f''$}
  \STATE Pick a non-mapped vertex $a$ adjacent to some already mapped vertex.
  \STATE Choose some $k \in \Large\cap_{c \in f''(V'')} I_c$ with $c$'s nad $I_c$'s defined
  analogously as in the proof of Theorem \ref{thm:ext}.
  \STATE Set $f'' := f'' \cup (a,k)$.
  \ENDWHILE
  \RETURN \TRUE
  \ELSE
  \RETURN \FALSE
  \ENDIF
\ENDFOR
\ENDIF
\end{algorithmic}
\end{algorithm}

\begin{proof}[Proof of Theorem \ref{thm:general}]
  Because of Theorem \ref{thm:ext} our Algorithm \ref{alg:general} is correct.
  We need to analyze its complexity.

  Line 1 computes all-pairs distances in time $O(|V||E|)$. 
  Line 8 takes $O(|V|^2)$ time and the code between lines 10 and 12 also.
  If $f'$ was rooted, we jump straight into the for-cycle and thus $O(|V|)$ factor
  is saved, otherwise we get additional factor $O(|V|)$. Thus we conclude the claim.
\qed \end{proof}

\paragraph*{Fixed range and the maximum possible range.}
We can also ask if it is possible to extend a given mapping to an $M$-Lipschitz mapping
with a given range. Naturally, one can also define \textsc{MaxRange-$M$-LipExt} asking
for the maximum possible range of extending mapping. Minimum variant can be approached
in the same manner. Formally, we have the following problem.

\computationproblem
{Partial $M$-Lipschitz mapping extension with a fixed range - \textsc{FixedRange-$M$-LipExt}}
{A connected graph $G=(V,E)$, a subset $V' \subseteq V$ with a function $f': V' \to \Z$ and $r \in \N$.}
{Does there exist an $M$-Lipschitz mapping $f$ of $G$ such that $f' \subseteq f$ and $\rng(f) = r$?}

Using Algorithm \ref{alg:general} as a subroutine we can first choose an
interval $[a,b]$ such that $a,b \in \N$ and $0 \in [a,b]$. Then we can choose
independently two different vertices $v_a,v_b \in V\setminus V'$ .
We can then set $f_{[a,b],v_a,v_b} := f' \cup (v_b,b) \cup (v_a,a)$. Of
course, $a$ (or $b$) can already be the image of some vertex in $V'$ and in
that case we choose only $v_b$ (or $v_a$).

For every choice of $[a,b]$ and $v_a, v_b$ we run Algorithm \ref{alg:general}
and see if $f_{[a,b],v_a,v_b}$ can be extended. This algorithm is correct and
since $\abs{a-b} \le |V|$ we deduce that  \textsc{FixedRange-$M$-LipExt}
can be solved in time $O(|V|^7)$. Furthermore, \textsc{MaxRange-$M$-LipExt} can be
solved in time $O(|V|^7 \cdot \log |V|)$ by combining the algorithm for
 \textsc{FixedRange-$M$-LipExt} with a binary search for a suitable $r$.

 Finally, we note that in the case of diameter being bounded we can easily modify
 these algorithms to run in time $O(|V|^5)$ for both \textsc{FixedRange-$M$-LipExt}
 and \textsc{MaxRange-$M$-LipExt}.

\subsection{Partial strong $M$-Lipschitz mappings}

We note that for \textsc{Strong $M$-LipExt} we can modify Algorithm \ref{alg:trees}
and \ref{alg:general}. However, more involved analysis is needed. If we are satisfied with a worse running time, we can conclude the following theorem by analysis similar
to that for \textsc{$M$-LipExt} in Theorem \ref{thm:parext}, using results of Feder and Hell \cite{feder1999list}.

\begin{theorem}
  \textsc{Strong $M$-LipExt} is solvable in polynomial time.
\end{theorem}

\section{Concluding remarks}

We initiated the study of algorithmic aspects of Lipschitz mappings of graphs. We showed that the problem of finding the maximum range and extending partial Lipschitz mapping are both solvable in
polynomial time.

We conclude with two open problems stemming from our research.

\paragraph*{Improving the complexity of extension problems.}

The following open problem naturally arises: Are there algorithms for \textsc{FixedRange-$M$-LipExt} and \textsc{MaxRange-$M$-LipExt}
  running in time asymptotically better than ours?

\paragraph*{Average range.}

We propose to settle the complexity of the following natural problem associated
with the average range. The problem \textsc{$M$-AvgRange} has a connected graph $G$
on input. We ask what is the average range of $M$-Lipschitz mappings of
$G$, i.e.\ $\avg_M(G)$? Even the case $M=1$ seems challenging.

We note that for several classes of graphs, e.g.\ paths, cycles, complete graphs,
complete bipartite graphs we showed  \cite{bok2018graph} a closed formulas implying that we can compute \textsc{$1$-AvgRange} for these classes efficiently.

\section*{Acknowledgments}

This research was supported by the Charles University Grant Agency, project GA UK 1158216
and by the Center of Excellence - ITI (P202/12/G061 of GA\v{C}R).

\bibliographystyle{plain}
\bibliography{bibliography}

\appendix

\section{Omitted theorems and proofs}

\begin{proof}{(of Lemma \ref{lem:strongchar})}
  Observe that in every strong $M$-Lipschitz mapping, all
  vertices are mapped to some number of the form $k \cdot M$, where $k \in \Z$.

  Consider a non-bipartite graph $G$ and a strong $M$-Lipschitz mapping $f$ of $G$.
  We recall that a graph is bipartite if and only if it does not contain a cycle of odd
    length as a subgraph.

  Therefore, $G$ contains some odd cycle $C$ with edges $v_1v_2,\ldots, v_lv_1$.
  Let us denote 
  $$e_i := f(v_{(i + 1 \textrm{ mod } l)}) - f(v_{(i \textrm{ mod } l)}), \forall i \in \{1,\ldots, n\}.$$
  We see that $e_i \in \{+M, -M\}$. Moreover, $\sum_{i=0}^n e_i = 0$ from the definition of~$e_i$. However, this sum has an odd number of summands
  and thus we get a contradiction.
\qed \end{proof}

\begin{proof}{(of Lemma \ref{lem:diam})}
The existence of $f$ with $\rng(f) > M \cdot (\diam(G) + 1)$ would imply the existence of a path
subgraph in $G$ with endpoints $u, v$ such that their images would be in distance $|f(u) -
f(v)| = M \cdot (\diam(G) + 1)$. However, that would mean that all paths between $u$ and
$v$ have to map to some connected subgraph of $\Z$ of size greater than
$M \cdot (\diam(G) + 1)$ which is a contradiction with the definition of diameter.
\qed \end{proof}

\begin{proof}{(of Theorem \ref{thm:diamstrong})}
  We can take Lipschitz mapping $f$ as in Theorem \ref{thm:diam}. However, we
  have to check if it is a strong $M$-Lipschitz mapping.

  Suppose that $f$ is not a strong $M$-Lipschitz mapping. That means that there exist
  two vertices $a, b \in V(G)$ such that $ab \in E(G)$ and $f(a) = f(b)$.
  Furthermore, from the definition of $f$, $d(r,a) = d(r,b)$. Define $l := d(r,a)$.

  From the definition of the distance, we get that there exist an $(r,a)$-path
  and $(r,b)$-path, both of length $l$. Since $G$ is bipartite, parts to which vertices
  belong have to alternate along the $(r,a)$-path and along the $(r,b)$-path as well.
  Additionally, parts to which $a$ and $b$ belong are determined by the parity
  of their distance from root. But that means that $a$ and $b$ belong to the same
  part. Since they are neighbors, we get a contradiction.

  Finally, we note that the previous argument works also in the case of $r$ being
  either the vertex $a$ or $b$.
\qed \end{proof}

\end{document}